\documentclass[11pt,reqno]{amsart}

\usepackage[T1]{fontenc}
\usepackage[utf8]{inputenc}
\usepackage{lmodern}
\usepackage{amsmath,amssymb,mathtools}
\usepackage{microtype}
\usepackage[colorlinks=true,linkcolor=blue,citecolor=blue,urlcolor=blue]{hyperref}

\newcommand{\R}{\mathbb{R}}

\newcommand{\HS}{\mathrm{HS}}
\newcommand{\vr}{\operatorname{vr}}

\newcommand{\E}{\mathbb{E}}

\theoremstyle{plain}
\newtheorem{theorem}{Theorem}[section]

\newtheorem{lemma}[theorem]{Lemma}

\theoremstyle{remark}
\newtheorem{remark}[theorem]{Remark}

\title[An improved volume-ratio bound]
{An Improved Volume Ratio Bound via Isotropic Positions}

\author[D. Galicer]{Daniel Galicer} \address{Universidad Torcuato Di Tella. Departamento de Matem\'aticas y Estad\'istica. IMAS--CONICET. Av. Figueroa Alcorta 7350, C1428 Buenos Aires, Argentina} \email{daniel.galicer@utdt.edu}

\author[M. Merzbacher]{Mariano Merzbacher}
\address{Departamento de Matem\'atica, Facultad de Ciencias Exactas y Naturales,
Universidad de Buenos Aires, Ciudad Universitaria, Pabell\'on I,
C1428EGA Buenos Aires, Argentina}
\email{mmerzbacher@dm.uba.ar}

\author[D. Pinasco]{Dami\'an Pinasco}
\address{Universidad Torcuato Di Tella. Departamento de Matem\'aticas y Estad\'istica. CONICET. Av. Figueroa Alcorta 7350, C1428 Buenos Aires, Argentina}
\email{dpinasco@utdt.edu}


\begin{document}

\begin{abstract} We show that, for every pair of convex bodies $K,L\subset\mathbb R^n$, \[ \operatorname{vr}(K,L)\leq C\sqrt{n\log(n+1)}. \]
The main point is to place $K$ and $L^\circ$
in isotropic position. We then consider a random orthogonal
image of $L$ and control the corresponding operator norm by combining
the isotropic mean-gauge estimate of Bizeul and Klartag with Letwin's
recent dimension-free bound for the third-moment parameter appearing
in their estimate.

Our result improves the bound $\operatorname{vr}(K,L)\leq C\sqrt n\,\log(n+1)$ proved by Giannopoulos and Hartzoulaki, which had remained the best general estimate for nearly two and a half decades.
\end{abstract}

\maketitle

\section{Introduction}

A recurrent problem in asymptotic convex geometry is to approximate a
convex body by affine images of another body. For convex bodies
$K,L\subset\R^n$, their volume ratio is defined by
\begin{equation}\label{eq:def-vr}
 \vr(K,L)=\inf\left\{
 \left(\frac{|K|}{|T(L)|}\right)^{1/n}:
 T(L)\subseteq K
 \right\},
\end{equation}
where $T:\R^n\to\R^n$ is an invertible affine map.
Here $|K|$ denotes $n$-dimensional Lebesgue measure. The quantity in
\eqref{eq:def-vr} is invariant under invertible affine
transformations of its two arguments: if $S$ and $T$ are invertible
maps, then
\begin{equation}\label{eq:independent-invariance}
 \vr(S(K),T(L))=\vr(K,L).
\end{equation}
It also satisfies the following multiplicative triangle inequality: \begin{equation*}\vr(K,L)\leq \vr(K,S)\vr(S,L). \end{equation*}
The classical volume ratio of a convex body $K\subset\R^n$ is the special case $\vr(K):=\vr(K,B_2^n).$ Equivalently, if $\mathcal E_K$ denotes an ellipsoid of maximal volume contained in $K$, then \[ \vr(K)=\left(\frac{|K|}{|\mathcal E_K|}\right)^{1/n}. \] Thus, $\vr(K)$ is an affine invariant that measures how efficiently $K$ can be approximated from within by an ellipsoid. By John's theorem, \[ \vr(K)\leq n, \] and the sharper estimate $\vr(K)\leq\sqrt n$ holds when $K$ is centrally symmetric; see, for instance, \cite{AGM,Ball,PisierVolume}.
Applying the multiplicative triangle inequality with the Euclidean ball as an intermediate body, and using Ball's sharp estimate for the classical volume ratio \cite{Ball} together with the corresponding outer volume-ratio estimate following from Barthe's reverse Brascamp--Lieb inequality \cite{Barthe}, we obtain \[ \vr(K,L) \leq \vr(K,B_2^n)\vr(B_2^n,L) \leq n. \]
John-type decompositions for maximal-volume positions of one convex
body inside another were established in \cite{GPT,GLMP}, leading to
linear-in-$n$ estimates for volume ratios and Banach--Mazur distances.

A major improvement was proved by Giannopoulos and Hartzoulaki
\cite{GH} in 2002:
\begin{equation}\label{eq:GH}
 \vr(K,L)\leq C\sqrt n\,\log(n+1)
 \qquad (K,L\subset\R^n).
\end{equation}
As usual, $C>0$ denotes throughout this article an absolute constant whose value may
change from line to line.

Their proof combines suitable positions of the two bodies with Chevet's
inequality and an unbalanced form of the $MM^*$ estimate, following an
idea of Rudelson \cite{Rudelson}. The estimate \eqref{eq:GH} had been the best general upper bound.

Note that the order $\sqrt n$ is unavoidable in general. For instance, the Euclidean ball
and a simplex give a volume ratio of this magnitude. More generally, building on Gluskin's random-polytope construction \cite{Gluskin}, Khrabrov \cite{Khrabrov} proved that every $n$-dimensional convex body $K$ admits a convex body $L$ with \[ \vr(K,L) \geq C\sqrt{\frac{n}{\log\log(n+2)}}. \]
In \cite{GMP}, we removed the iterated logarithm by showing that $L$ may be chosen as a random polytope satisfying
\[
\operatorname{vr}(K,L)\geq C\sqrt n.
\]
A related lower bound for the volume ratio between projections of convex bodies was obtained in \cite{GMPLitvak}.

For many natural and structurally diverse convex bodies $K$, including Euclidean balls, simplices, unconditional bodies, unit balls of unitarily invariant norms, and unit balls of projective and injective tensor products of $\ell_p$-spaces, one has \[ \vr(K,L)\leq C\sqrt n \] uniformly over all convex bodies $L$; see, for instance, \cite{GMP}. Despite this evidence, it remained open whether the logarithmic loss in the general estimate of Giannopoulos and Hartzoulaki \cite{GH} could be reduced. Our main result replaces this logarithmic factor by its square root.

\begin{theorem}\label{thm:main}
There exists a universal constant $C>0$ such that, for every $n\geq2$ and
every pair of convex bodies $K,L\subset\R^n$,
\begin{equation*}
 \vr(K,L)\leq C\sqrt{n\log(n+1)}.
\end{equation*}
\end{theorem}

We first reduce the problem to the centrally
symmetric case. For a centrally symmetric convex body
$W\subset\R^n$, let
\[
   X_W=(\R^n,\|\cdot\|_W)
   \qquad\text{and}\qquad
   \ell_2^n=(\R^n,|\cdot|_2),
\]
so that $W$ is the unit ball of $X_W$.

The argument relies on a careful choice of positions in the classical
random-operator method. Given centrally symmetric convex bodies $K,L\subset\mathbb{R}^n$,
let $U$ be Haar-distributed on the orthogonal group $O(n)$. We use
the following classical orthogonal form of Chevet's inequality,
obtained by combining Chevet's Gaussian inequality with the
Marcus--Pisier comparison between Gaussian and orthogonal averages;
see \cite[Proposition~43.6]{TomczakJaegermannBook} or \cite[Section~9.4]{AGM} for a modern exposition:
\begin{equation}\label{eq:orthogonal-chevet}
\E\bigl\|U:X_L\to X_K\bigr\|
\leq
C\Bigl(
R(L)M(K)+R(K^\circ)M(L^\circ)
\Bigr).
\end{equation}
Here
\[
R(W):=\bigl\|\operatorname{id}:X_W\to\ell_2^n\bigr\|
=\max_{x\in W}|x|_2
\]
denotes the Euclidean radius of $W$, while
\[
M(W):=\int_{S^{n-1}}|x|_W\,d\sigma(x)
\]
is its spherical mean gauge, with $\sigma$ denoting the normalized
rotation-invariant measure on $S^{n-1}$.

The key observation is to place $K$ and $L^\circ$ in  isotropic position. This choice is tailored to
the structure of the orthogonal Chevet inequality: it simultaneously
controls the radius terms $R(L)$ and $R(K^\circ)$ and the mean-gauge
terms $M(K)$ and $M(L^\circ)$.

\section{Proof of the main theorem}\label{sec:proof}

We begin by recalling some standard definitions.
A convex body $W\subset\R^n$ is said to be in isotropic position if its barycenter is the origin, $|W|=1$, and \begin{equation*} \int_W \langle x,\theta\rangle^2\,dx =L_W^2|\theta|_2^2 \qquad (\theta\in\R^n),
\end{equation*}
where $L_W$ is the isotropic constant of $W$. We shall use the well-known universal lower bound

\begin{equation}\label{eq:isotropic-lower} L_W\geq C;
\end{equation}
see, for instance, \cite{BGVV}. For a convex body $W$ containing the origin in its interior, we write \[ W^\circ =\{y\in\R^n:\langle x,y\rangle\leq 1 \text{ for every }x\in W\}. \]

We start by reducing the problem to the centrally symmetric case. This argument is standard and was already used by Giannopoulos and Hartzoulaki in \cite{GH} (see also \cite{GMP}). We state it as a lemma and include the proof for completeness.
\begin{lemma}[Reduction to the centrally symmetric case] \label{lem:symmetric-reduction} For every pair of convex bodies $K,L\subset\R^n$, there exist centrally symmetric convex bodies $K_0,L_0\subset\R^n$ such that \[ \vr(K,L)\leq C\,\vr(K_0,L_0). \] \end{lemma}

\begin{proof}
By affine invariance, we may assume that the barycenter of $K$ is the
origin and that $0\in L$. Set
\[
   K_0:=K\cap(-K)
   \qquad\text{and}\qquad
   L_0:=L-L.
\]
Both bodies are centrally symmetric. The Rogers--Shephard inequalities
\cite{RS} give
\[
   |K_0|^{1/n}\geq \frac12 |K|^{1/n},
   \qquad
   |L_0|^{1/n}\leq 4|L|^{1/n}.
\]
Since $K_0\subseteq K$ and $L\subseteq L_0$, it follows that
\[
   \vr(K,K_0)\leq 2
   \qquad\text{and}\qquad
   \vr(L_0,L)\leq 4.
\]
Therefore, applying the multiplicative triangle inequality twice,
\[
   \vr(K,L)
   \leq
   \vr(K,K_0)\vr(K_0,L_0)\vr(L_0,L)
   \leq 8\,\vr(K_0,L_0).\qedhere
\]
\end{proof}

The estimate needed below follows by combining the isotropic
mean-gauge bound of Bizeul and Klartag with Letwin's recent
dimension-free control of the third-moment parameter. For $n\geq2$,
let
\begin{equation*}
   \kappa_n
   :=
   \sup
   \left\|
      \E\bigl[\langle X,\theta\rangle X\otimes X\bigr]
   \right\|_{\HS},
\end{equation*}
where the supremum runs over all isotropic log-concave random vectors
$X$ in $\R^n$ and all $\theta\in S^{n-1}$.

\begin{lemma} [Bizeul-Klartag + Letwin]
\label{lem:isotropic-M}
Let $W\subset\R^n$ be a convex body in volume-one isotropic position.
Then
\begin{equation}\label{eq:M-isotropic-final}
   M(W)
   \leq
   C\sqrt{\frac{\log(n+1)}{n}}.
\end{equation}
\end{lemma}

\begin{proof}
By the mean-gauge estimate of Bizeul and Klartag
\cite[Lemma~4.1]{BizeulKlartag}, after rescaling from their
covariance-one normalization to our volume-one isotropic
normalization, we have
\begin{equation}\label{eq:M-isotropic-kappa}
   M(W)
   \leq
   C\frac{\kappa_n\sqrt{\log(n+1)}}{L_W\sqrt n}.
\end{equation}
On the other hand, Letwin proves, as part of the proof of
\cite[Theorem~1.1]{Letwin}, that
\begin{equation}\label{eq:kappa-bound}
   \kappa_n\leq 2\sqrt2.
\end{equation}
Combining these estimates with the universal lower bound $L_W\geq C$
proves \eqref{eq:M-isotropic-final}.
\end{proof}

We are now ready to prove our main theorem.

\begin{proof}[Proof of Theorem~\ref{thm:main}]
By Lemma~\ref{lem:symmetric-reduction}, it is enough to consider
centrally symmetric bodies $K$ and $L$.

Choosing a linear position for $L^\circ$ is equivalent to choosing the
corresponding position for $L$, since
\[
   \bigl((T^*)^{-1}L\bigr)^\circ=T(L^\circ),
   \qquad T\in GL(n).
\]
Thus, by the affine invariance of the volume ratio
\eqref{eq:independent-invariance}, we may place $K$ and $L^\circ$ in
isotropic position. In
particular,
\begin{equation}\label{eq:normalizations-proof}
   |K|=|L^\circ|=1.
\end{equation}

For every isotropic body $W$, the standard inclusion
\[
   CL_W B_2^n\subseteq W
\]
holds; see, for instance, \cite{AGM}. Passing to polars and using the
universal lower bound \eqref{eq:isotropic-lower} $L_W\geq C$, we obtain
\[
   W^\circ\subseteq C B_2^n.
\]
Applying this first to $W=K$ and then to $W=L^\circ$ gives
\begin{equation}\label{eq:radius-control}
   R(K^\circ)\leq C,
   \qquad
   R(L)=R\bigl((L^\circ)^\circ\bigr)\leq C.
\end{equation}
Moreover, by \eqref{eq:M-isotropic-final},
\begin{equation}\label{eq:M-control}
   M(K)+M(L^\circ)
   \leq
   C\sqrt{\frac{\log(n+1)}{n}}.
\end{equation}

Let $U$ be Haar-distributed on $O(n)$. Applying Chevet's inequality \eqref{eq:orthogonal-chevet}, and using
\eqref{eq:radius-control} and \eqref{eq:M-control}, we obtain
\begin{align*}
   \E\bigl\|U:X_L\to X_K\bigr\|
   &\leq
   C\bigl(
      R(L)M(K)+R(K^\circ)M(L^\circ)
   \bigr)\\
   &\leq
   C\sqrt{\frac{\log(n+1)}{n}}.
\end{align*}
Consequently, there exists $U\in O(n)$ such that
\begin{equation}\label{eq:a-bound}
   a_{L,K}:=\bigl\|U:X_L\to X_K\bigr\|
   \leq
   C\sqrt{\frac{\log(n+1)}{n}}.
\end{equation}
Since $U(L)\subseteq a_{L,K}K$, we have \[ a_{L,K}^{-1}U(L)\subseteq K. \] Moreover, since $U$ is orthogonal,  $|\det U|=1$ and hence \[ |a_{L,K}^{-1}U(L)|=a_{L,K}^{-n}|L|. \] Therefore, by the definition of the volume ratio and \eqref{eq:normalizations-proof}, \begin{equation}\label{eq:vr-before-santalo} \vr(K,L) \leq \left(\frac{|K|}{|a_{L,K}^{-1}U(L)|}\right)^{1/n} = a_{L,K}\left(\frac{|K|}{|L|}\right)^{1/n} = a_{L,K}\,|L|^{-1/n}. \end{equation}
Finally, the reverse Santaló inequality of Bourgain and Milman
\cite{BM} gives
\[
   \bigl(|L||L^\circ|\bigr)^{1/n}
   \geq
   C\,|B_2^n|^{2/n}
   \geq
   \frac{C}{n}.
\]
Since $|L^\circ|=1$, it follows that
\begin{equation}\label{eq:L-volume}
   |L|^{-1/n}\leq Cn.
\end{equation}
Combining \eqref{eq:a-bound}, \eqref{eq:vr-before-santalo}, and
\eqref{eq:L-volume}, we conclude that
\[
   \vr(K,L)
   \leq
   Cn\sqrt{\frac{\log(n+1)}{n}}
   =
   C\sqrt{n\log(n+1)}.
\]
The general case follows from Lemma~\ref{lem:symmetric-reduction},
after adjusting the universal constant.
\end{proof}

\begin{remark}[On the remaining logarithm]\label{rem:log} The factor $\sqrt{\log(n+1)}$ arises naturally in the Haar-average argument used above. Indeed, let \[ K=L^\circ=\frac12 B_\infty^n, \qquad L=2B_1^n. \] If $U=(u_{ij})\in O(n)$, then $\bigl\|U:X_L\to X_K\bigr\| = 4\max_{i,j}|u_{ij}|.$
Writing $W_n=\max_{i,j}|u_{ij}|$, Jiang \cite{Jiang} proved that
$\sqrt{n/\log n}\,W_n$ converges to $2$ in probability; together
with the corresponding tail estimates, this yields
$$\E W_n \asymp
\sqrt{\frac{\log(n+1)}{n}}.$$
On the other hand, \[ \vr(B_\infty^n,B_1^n)\leq C\sqrt n, \] as follows, for instance, from the Dvoretzky--Rogers theorem \cite[Theorem~5A]{DvoretzkyRogers}; see also \cite[Proof of Theorem~1.3]{GMPsimplex} or \cite{PelczynskiSzarek}. Thus the Haar first moment is not sharp even for this particular pair. This example gives no indication, however, as to the optimal general order of $\vr(K,L)$. It remains open whether the logarithmic factor in our bound can be removed, or whether there exist pairs of convex bodies $K,L\subset\R^n$ for which $\vr(K,L)$ grows faster than $\sqrt n$. \end{remark}

\subsection*{Note added}
After the first version of this article was completed, we became aware of two remarkable recent preprints by Bizeul \cite{BizeulMMstar} and by Paouris and Pathak \cite{PaourisPathak}. Neither of these works considers the volume-ratio problem studied here, and our estimate does not appear explicitly in either of them. Nevertheless, their results can be combined with the orthogonal Chevet argument of the present paper to provide alternative routes to our main estimate.

Indeed, using Bizeul's recent Theorem~1.1 in \cite{BizeulMMstar},
one could instead place both $K$ and $L$ in isotropic position.
The term $M(K)$ can then be estimated exactly as in
Lemma~\ref{lem:isotropic-M}, while Theorem~1.1 of
\cite{BizeulMMstar} provides the required estimate for
$M(L^\circ)=M^*(L)$. In this case one uses the standard isotropic
radius estimate $
R(L) \leq Cn$, while $R(K^\circ)\leq C$ as before. Inserting these estimates
into the same orthogonal Chevet argument again yields the order
$\sqrt{n\log(n+1)}$.

There is also a different route using the results of Paouris and Pathak.
Namely, one may place $K$ and $L^\circ$ in Bobkov's maximal Gaussian measure position, normalized so that
$|K|=|L^\circ|=1$. By Lemma~2.2 of \cite{PaourisPathak}, this gives
\[
R(K^\circ)+R(L)\leq C.
\]
On the other hand, applying Theorem~6.1 of \cite{PaourisPathak} to
$K^\circ$ and $L$, respectively, yields
\[
M(K)+M(L^\circ)
\leq
C\sqrt{\frac{\log(n+1)}{n}}.
\]
These are precisely the estimates needed in the orthogonal Chevet inequality.
Combining them with the same Chevet and reverse Santal\'o argument used above
again gives the order $\sqrt{n\log(n+1)}$.

These two preprints appeared after the present manuscript had been completed and illustrate the remarkably rapid recent progress in asymptotic convex geometry.

\section*{Use of AI tools}

No AI tools were used in the development, verification, or validation of the mathematical ideas, arguments, proofs, or results presented in this article. Generative AI tools were used solely for language editing and proofreading.

\section*{Acknowledgments}

The authors are grateful to Alexander Litvak for his helpful comments concerning appropriate references for some of the results used in the paper.


\begin{thebibliography}{99}

\bibitem{AGM}
S.~Artstein-Avidan, A.~Giannopoulos and V.~D.~Milman,
\emph{Asymptotic Geometric Analysis, Part I},
Math. Surveys Monogr. 202, American Mathematical Society, Providence, RI,
2015.


\bibitem{Ball}
K.~Ball,
Volume ratios and a reverse isoperimetric inequality,
\emph{J. London Math. Soc. (2)} \textbf{44} (1991), 351--359.

\bibitem{Barthe}
F.~Barthe,
On a reverse form of the Brascamp--Lieb inequality,
\emph{Invent. Math.} \textbf{134} (1998), 335--361.



\bibitem{BizeulMMstar}
P.~Bizeul,
Optimal $MM^*$ bounds for convex bodies,
arXiv:2607.29458 (2026).

\bibitem{BizeulKlartag}
P.~Bizeul and B.~Klartag,
Distances between non-symmetric convex bodies: optimal bounds up to
polylog,
\emph{arXiv preprint arXiv:2510.20511}, 2026.

\bibitem{BM}
J.~Bourgain and V.~D.~Milman,
New volume ratio properties for convex symmetric bodies in $\R^n$,
\emph{Invent. Math.} \textbf{88} (1987), 319--340.




\bibitem{BGVV} S.~Brazitikos, A.~Giannopoulos, P.~Valettas, and B.-H.~Vritsiou, \emph{Geometry of Isotropic Convex Bodies}, Mathematical Surveys and Monographs, vol.~196, American Mathematical Society, Providence, RI, 2014.
\bibitem{DvoretzkyRogers}
A.~Dvoretzky and C.~A.~Rogers,
Absolute and unconditional convergence in normed linear spaces,
\emph{Proc. Natl. Acad. Sci. USA} \textbf{36} (1950),
no.~3, 192--197.

\bibitem{GMPLitvak}
D.~Galicer, A.~E.~Litvak, M.~Merzbacher, and D.~Pinasco,
On the volume ratio of projections of convex bodies,
\emph{J. Funct. Anal.} \textbf{286} (2024), no.~3, 110242.

\bibitem{GMP}
D.~Galicer, M.~Merzbacher and D.~Pinasco,
Asymptotic estimates for the largest volume ratio of a convex body,
\emph{Rev. Mat. Iberoam.} \textbf{37} (2021), 2347--2372.

\bibitem{GMPsimplex}
D.~Galicer, M.~Merzbacher, and D.~Pinasco,
The minimal volume of simplices containing a convex body,
\emph{J. Geom. Anal.} \textbf{29} (2019),
no.~1, 717--732.



\bibitem{GH}
A.~Giannopoulos and M.~Hartzoulaki,
On the volume ratio of two convex bodies,
\emph{Bull. London Math. Soc.} \textbf{34} (2002), 703--707.

\bibitem{GPT}
A.~Giannopoulos, I.~Perissinaki and A.~Tsolomitis,
John's theorem for an arbitrary pair of convex bodies,
\emph{Geom. Dedicata} \textbf{84} (2001), 63--79.

\bibitem{Gluskin} E.~D.~Gluskin, The diameter of the Minkowski compactum is approximately equal to $n$, \emph{Funct. Anal. Appl.} \textbf{15} (1981), no.~1, 57--58.

\bibitem{GLMP}
Y.~Gordon, A.~E.~Litvak, M.~Meyer and A.~Pajor,
John's decomposition in the general case and applications,
\emph{J. Differential Geom.} \textbf{68} (2004), 99--119.

\bibitem{Jiang}
T.~Jiang,
Maxima of entries of Haar distributed matrices,
\emph{Probab. Theory Related Fields} \textbf{131} (2005), 121--144.


\bibitem{Khrabrov}
A.~Khrabrov,
Generalized volume ratios and the Banach--Mazur distance,
\emph{Math. Notes} \textbf{70} (2001), 838--846.


\bibitem{Letwin}
B.~Letwin,
The KLS constant is $O(\log^{1/4} n)$,
\emph{arXiv preprint arXiv:2607.24164}, 2026.


\bibitem{PaourisPathak}
G.~Paouris and R.~Pathak,
Optimal mean width and metric entropy estimates for convex bodies,
arXiv:2607.29522 (2026).

\bibitem{PelczynskiSzarek}
A.~Pe{\l}czy\'nski and S.~J.~Szarek,
On parallelepipeds of minimal volume containing a convex symmetric
body in $\R^n$,
\emph{Math. Proc. Cambridge Philos. Soc.} \textbf{109} (1991),
no.~1, 125--148.

\bibitem{PisierVolume}
G.~Pisier,
\emph{The Volume of Convex Bodies and Banach Space Geometry},
Cambridge Tracts in Mathematics, vol.~94,
Cambridge University Press, Cambridge, 1989.

\bibitem{RS}
C.~A.~Rogers and G.~Shephard,
The difference body of a convex body,
\emph{Arch. Math. (Basel)} \textbf{8} (1957), 220--233.


\bibitem{Rudelson}
M.~Rudelson,
Distances between non-symmetric convex bodies and the $MM^*$-estimate,
\emph{Positivity} \textbf{4} (2000), 161--178.


\bibitem{TomczakJaegermannBook}
N.~Tomczak-Jaegermann,
\emph{Banach--Mazur Distances and Finite-Dimensional Operator Ideals},
Pitman Monographs and Surveys in Pure and Applied Mathematics,
vol.~38, Longman Scientific \& Technical, Harlow;
copublished in the United States with John Wiley \& Sons, Inc.,
New York, 1989.

\end{thebibliography}
\end{document}